\documentclass[11pt]{article}
\usepackage{amsthm,enumerate}
\renewcommand{\baselinestretch}{1.2}
\newcommand{\dated}{\mbox{} \hfill {\small [{\tt \today}]}} \usepackage{amsmath,amssymb,amsfonts,diagrams}
\newenvironment{keywords}{\noindent\small {\it Keywords\/}:}{\vskip 4pt}
\newenvironment{classification}{\noindent\small 2010 {\it Mathematics Subject
Classification\/}:}{\vskip 12pt}

\newcommand{\posints}{{\mathbb N}}

\newcommand{\free}{{\mathbb F}}

\newcommand{\cstar}{{C^\ast}}

\newcommand{\A}{{\mathfrak A}}

\newcommand{\Hilbert}{{\mathfrak H}}

\usepackage{amsthm,enumerate}
\theoremstyle{plain}
\newtheorem{theorem}{Theorem}[section]
\newtheorem{lemma}[theorem]{Lemma}
\newtheorem{corollary}[theorem]{Corollary}
\newtheorem{proposition}[theorem]{Proposition}
\theoremstyle{definition}
\newtheorem{definition}[theorem]{Definition}
\theoremstyle{remark}

\newtheorem*{rems}{Remarks}
\newtheorem*{exs}{Examples}

\newenvironment{items}{\begin{enumerate}[\rm (i)]}{\end{enumerate}}
\newenvironment{alphitems}{\begin{enumerate}[\rm (a)]}{\end{enumerate}}

\title{Operator ultra-amenability}
\author{\textit{Brian E.\ Forrest}\thanks{Research supported by NSERC.} \and \textit{Volker Runde}\thanks{Research supported by NSERC.} \and \textit{Kyle Schlitt}}
\date{}
\begin{document}
\maketitle
\begin{abstract}
Extending M.\ Daws' definition of ultra-amenable Banach algebras, we introduce the notion of operator ultra-amenability for completely contractive Banach algebras. For a locally compact group $G$, we show that the operator ultra-amenability of $A(G)$ imposes severe restrictions on $G$. In particular, it forces $G$ to be a discrete, amenable group with no infinite abelian subgroups. For various classes of such groups, this means that $G$ is finite.
\end{abstract}
\section*{Introduction}
In (\cite{Joh}), B.\ E.\ Johnson characterized the amenable locally compact groups $G$ in terms of a cohomological triviality condition of their group algebras $L^1(G)$. This cohomological triviality condition can be extended to arbitrary Banach algebras and defines the class of amenable Banach algebras.
\par
In (\cite{Daw}), M.\ Daws defined---motivated by his research with the second named author in (\cite{DR})---the notion of ultra-amenability of a Banach algebra. Given a Banach algebra $\A$ and an ultrafilter $\mathcal U$ over an arbitrary index set, the ultrapower $(\A)_{\mathcal U} $ is again a Banach algebra (\cite[Proposition 3.1(i)]{Hei}). Consequently, Daws defined $\A$ to be ultra-amenable if $(\A)_{\mathcal U}$ is amenable in the sense of \cite{Joh} for every ultrafilter $\mathcal U$ over any any index set. Daws proved that a $\cstar$-algebra is ultra-amenable if and only if it is subhomogeneous and that, for a locally compact group $G$ satisfying certain properties, the group algebra $L^1(G)$ is ultra-amenable only if $G$ is finite (\cite[Theorem 5.17]{Daw}), e.g., if $G$ is abelian, compact, or discrete (\cite[Theorems 5.9 and 5.11]{Daw}). He strongly suspected that $L^1(G)$ is ultra-amenable if and only $G$ is finite for \emph{every} locally compact group $G$.
\par
In this note, we introduce a notion of ultra-amenability in the operator space context and focus, in particular, on the Fourier algebra of a locally compact group.
\par
This paper is part of the third named author's PhD thesis under the second author's supervision.
\begin{sloppy} \section{Operator ultra-amenability for completely contractive Banach algebras} \end{sloppy}
In \cite{Rua1}, Z.-J.\ Ruan initiated the theory of abstract operator spaces: these spaces can be completely isometrically be represented as subspace of $\mathcal{B}(\Hilbert)$ for some Hilbert space $\Hilbert$. It is straightforward that the ultraproduct construction in the Banach space category carries over to the category of operator spaces (\cite[Section 10.3]{ER}). In \cite{Rua2}, Ruan introduced the notion of a \emph{completely contractive Banach algebra}: this is an algebra equipped with an operator space structure such that multiplication is a completely contractive bilinear map. Consequently, Ruan defined the notion of an \emph{operator amenable, completely contractive Banach algebra} by requiring the bounded derivations in Johnson's original definition of an amenable Banach algebra to be \emph{completely bounded}.
\par
If $\A$ is a completely contractive Banach algebra, and $\mathcal U$ is an ultrafilter over an arbitrary index set, then $(\A)_{\mathcal U}$ is also a completely contractive Banach algebra (as is easy to see). 
\par 
We thus define:
\begin{definition}
Let $\A$ be a completely contractive Banach algebra. We say that $\A$ is operator ultra-amenable if the ultrapower $(\A)_{\mathcal U}$ is operator amenable for any ultrafilter $\mathcal U$ over an arbitrary index set.
\end{definition}
\par
There are various canonical functors from the category of Banach spaces into the category of operator spaces. One of them is the $\max$ functor that assigns to each Banach space the \emph{largest} operator structure there is on it. By \cite[Proposition 1.5]{ER}, a Banach algebra $\A$ is amenable if and only if the completely contractible Banach algebra $\max \A$ is operator amenable. 
\par
We shall show that this is compatible with the ultraproduct construction:
\begin{proposition} \label{prop1}
Let $E$ be a Banach space and let $\mathcal{U}$ be an ultrafilter over any index set. Then the identity map on $(E)_{\mathcal{U}}$ induces a completely isomorphic isomorphism from $( \max E)_{\mathcal U}$ to $\max (E)_{\mathcal{U}}$.
\end{proposition}
\begin{proof}
Let $S$ be the closed unit ball of $E$. It is obvious that $E$ is a quotient of $\ell^1(S)$. Both the ultrapower contruction and the $\max$ functor are compatible with taking quotients. It is thus sufficient to show that the canonical map from $\max (\ell^1(S))_{\mathcal{U}}$ to $(\max \ell^1(S))_{\mathcal{U}}$ is a completely isomorphic isomorphism.
\par
This, however, is easily accomplished by adapting the proof of (\cite[Proposition 10.3.8]{ER}).
\end{proof}
\begin{corollary} \label{maxcor}
Let $\A$ be a Banach algebra. Then $\A$ is ultra-amenable if and only if $\max \A$ is operator ultra-amenable. 
\end{corollary}
\par
In \cite{Rua2}, Ruan showed that a $\cstar$-algebra is operator amenable if and only if it amenable. We show that the same holds true for operator ultra-amenability:
\begin{proposition}
For a $\cstar$-algebra $\A$, the following are equivalent:
\begin{items}
\item $\A$ is operator ultra-amenable;
\item $\A$ is ultra-amenable;
\item $\A$ is subhomogeneous.
\end{items}
\end{proposition}
\begin{proof}
(ii) $\Longleftrightarrow$ (iii) is from \cite{Daw} (see the corrigendum).
\par
(ii) $\Longrightarrow$ (i) is trivial in view of Corollary \ref{maxcor}.
\par
For (i) $\Longrightarrow$ (ii), let $\A$ be an operator ultra-amenable $\cstar$-algebra, and let $\mathcal U$ be an ultrafilter over an arbitrary index set. Then $(\A)_{\mathcal U}$ is operator amenable. However, $(\A)_{\mathcal U}$ is also  $\cstar$-algebra by \cite[Proposition 3.1(ii)]{Hei} and is therefore amenable by \cite[Theorem 5.1]{Rua2}. As $\mathcal U$ was arbitrary, $\A$ is ultra-amenable.
\end{proof}
\section{The case of the Fourier algebra}
It is fair to say that the notion of operator amenability was introduced with the Fourier algebra in mind. In \cite{FR}, the first and the second named author showed that, for a locally compact group $G$, its Fourier algebras $A(G)$ is amenable only $G$ is the finite extension of an abelian group. On the other hand, Ruan showed in \cite{Rua2} that $A(G)$ is \emph{operator amenable} if and only if $G$ is amenable---a much less restrictive property.
\par 
We now turn to the question for which locally compact groups $G$, the Fourier algebra $A(G)$ is operator ultra-amenable. In view of \cite{Daw}, it is very plausible to conjecture that, for any locally compact group $G$, the operator ultra-amenability of $A(G)$ would force $G$ to be finite. We shall not be able to prove this conjecture in full generality, but still establish it for important special cases.
\par
We start with a simple hereditary property:
\begin{lemma} \label{lem}
Let $G$ be a locally compact group such that $A(G)$ is operator ultra-amenable. Then $A(H)$ is operator ultra-amenable for every closed subgroup $H$ of $G$.
\end{lemma}
\begin{proof}
Let $H$ be a closed subgroup of $G$. Then the restriction map
\[
  A(G) \to A(H), \quad f \mapsto f |_H
\]
is a complete quotient map (and an algebra homomorphism).
\par
Let $\mathcal U$ be any ultrafilter over an arbitrary index set, so that $(A(G))_{\mathcal U}$ is operator amenable. Since the ultrafilter construction preserves (complete) quotient maps, the completely contractive Banach algebra $(A(H))_{\mathcal U}$ is also operator amenable. As $\mathcal U$ was arbitrary, this means that $A(H)$ is operator ultra-amenable. 
\end{proof}
\begin{corollary} \label{cor}
Let $G$ be a locally compact group such that $A(G)$ is operator ultra-amenable. Then every abelian subgroup of $G$ is finite.
\end{corollary}
\begin{proof}
Let $H$ be an abelian subgroup of $G$, and suppose without loss of generality that it is closed in $G$. As the Fourier transform between $A(H)$ and $L^1(\hat{H})$ is a completely isometric isomorphism, it follows that $\hat{H}$---being abelian---has to be finite by (\cite[Theorem 5.9]{Daw}) as has $H$.
\end{proof}
\par
We can now collect (quite) a few examples for of locally compact groups for which $A(G)$ is \emph{not} operator ultra-amenable:
\begin{proposition} \label{prop2}
Let $G$ be a locally compact group and let $H$ be an infinite closed subgroup of $G$ such that either of the following holds:
\begin{alphitems}
\item $H$ is compact;
\item $H$ is connected.
\end{alphitems}
Then $A(G)$ is not operator ultra-amenable.
\end{proposition}
\begin{proof}
Both of these cases can be shown to follow from Corollary \ref{cor}.
\par
Case (a): This  follows from Corollary \ref{cor} and from a result by E.\ I.\ Zelmanov, which asserts that every infinite compact group contains an infinite abelian subgroup (see \cite{Zel}).
\par
Case (b): Assume that $A(G)$ is operator ultra-amenable. Then this is also true for $A(H)$. As $H$ is a connected group, there are, by \cite[Theorem 4.13]{MZ}, abelian subgroups $H_1, \ldots, H_n$ of $H$, as well as a compact subgroup $K$ of $H$, such that
\[
  H_1 \times \cdots \times H_n \times K \to H, \quad (x_1, \ldots, x_n, y) \mapsto x_1 \cdots x_n y
\]
is a homeomorphism. Corollary \ref{cor} and Proposition \ref{prop2}(a) thus yield that $H_1, \ldots, H_n$, and $K$ must be all be finite subgroups of $H$, so that ultimately $H$ is finite, which is a contradiction. 
\end{proof}
\begin{sloppy} \begin{theorem}
Let $G$ be a locally compact group such that $A(G)$ operator ultra-amenable. Then $G$ is discrete and amenable, and contains no infinite abelian subgroup.
\end{theorem} \end{sloppy}
\begin{proof} 
That $G$ has no infinite abelian subgroup is, of course, the statement of Corollary \ref{cor}. That $G$ is amenable follows from fact that $A(G)$ is also operator amenable (see \cite{Rua2}). 
\par
To see that $G$ is discrete, we first observe that Proposition \ref{prop2}(b) forces $G$ to be totally disconnected. It follows from the Gleason--Yamabe Theorem (see \cite{Yam}), that every totally disconnected locally compact group has a neighborhood base of the identity consisting of open, compact subgroups. Since any such subgroup must be finite by Proposition \ref{prop2}(a), this forces $G$ to be discrete.
\end{proof}
\par
In view of this theorem, we shall suppose for the remainder of the paper that all groups considered are discrete.
\par
Historically, O.\ Yu.\ \v{S}midt conjectured that every infinite group contained an infinite abelian subgroup. An affirmative answer to \v{S}midt's conjecture would thus have allowed us to show that that operator ultra-amenability of $A(G)$ would indeed force the group $G$ to be finite. However, S.\ I.\ Adian and P.\ S.\ Novikov (see \cite{Adi}) showed that \v{S}midt's conjecture is false, i.e., there are infinite groups without infinite abelian subgroups. Nevertheless, there are a number of important classes of groups $G$ for which we shall be able to show that the operator ultra-amenability of $A(G)$ does indeed imply that $G$ is finite. 
\par 
Recall that a group called
\begin{itemize}
\item \emph{periodic} if each of its element has finite order and 
\item \emph{locally finite} if each of its finite subsets generates a finite subgroup. 
\end{itemize}
\par
Clearly, every locally finite group is periodic whereas the converse is obviously false. One of the fundamental properties of locally finite groups is that if $G$ is infinite and locally finite, then $G$ contains an infinite abelian subgroup. This was established independently by P.\ Hall and C.\ R.\ Kulatilaka (\cite{HK}) as well as M.\ I.\ Kargarpolov (\cite{Kar}). As the operator amenability of $A(G)$ for every group $G$, forces $G$ to be periodic by Corollary \ref{cor}, this means:
\begin{sloppy} \begin{proposition} \label{prop3}
Let $G$ be a locally finite group such that $A(G)$ is operator ultra-amenable. Then $G$ is finite.  
\end{proposition} \end{sloppy}
\par
Amongst all discrete amenable groups, perhaps the simplest and best understood are the \emph{elementary amenable groups}. Recall that the elementary amenable groups form the smallest class  $\mathcal{E}$ of groups $G$ such that: 
\begin{alphitems} 
\item $\mathcal{E}$ contains all finite and all abelian groups;
\item if $G \in \mathcal{E}$ and $H$ is a group isomorphic to $G$, then $H \in \mathcal{E}$;
\item $\mathcal{E}$ is closed under forming subgroups, quotients, and extensions;
\item $\mathcal{E}$ is closed under directed unions.
\end{alphitems}
\par
Recall that a group is called \emph{linear} if it is isomorphic to a subgroup of $\mathrm{GL}(n,\free)$ for some field $\free$. Suppose that $G$ is linear and amenable, and let $H$ be a finitely generated subgroup of $G$. As a consequence, of the main result of \cite{Tit}, $H$ either contains a copy of the free group on two generators or has a solvable subgroup of finite index. Since $H$ is amenable, only the latter is possible. Now, solvable groups are elementary amenable, as are finite extensions of solvable groups. Since $G$ is the directed union of all of its finitely generated subgroups, it follows that $G$ is elementary amenable. Moreover, C.\ Chou (\cite[Theorem 2.3]{Cho}) showed that a periodic elementary amenable group is always locally finite. 
\par
We thus obtain the following corollary to Proposition \ref{prop3}:
\begin{corollary} \label{cor2}
Let $G$ be a linear group or an elementary amenable group such that $A(G)$ is operator ultra-amenable. Then $G$ is finite.  
\end{corollary}
\par
Finally, we say that a group $G$ has \emph{polynomial growth} if, for each finite set $F \subset G$, there is $p \in \posints$ such that $|F^n|=O(n^p)$ for all $n \in \posints$. (Here, $F^n$ denotes the set $\{ x_1 \cdots x_n : x_1, \ldots, x_n \in H \}$.) 
\par 
We have:
\begin{proposition} \label{prop4}
Let $G$ be a group with polynomial growth such that $A(G)$ is operator ultra-amenable. Then $G$ is finite.  
\end{proposition}
\begin{proof} Let $H$ be a finitely generated subgroup of $G$. Then $H$ also has polynomial growth. It follows from a result of M.\ L.\ Gromov \cite{Gro} that $H$ has a nilpotent subgroup $N$ of finite index in $H$. If $N$ were infinite, it would contain an infinite abelian subgroup, which is impossible by Lemma \ref{lem}. Consequently $N$ must be finite, as must be $H$. This shows that $G$ is locally finite. We conclude from Proposition \ref{prop3} that $G$ is finite. 
\end{proof}
\renewcommand{\baselinestretch}{1.0}
\renewcommand{\baselinestretch}{1.2}
\dated
\vfill
\begin{tabbing} 
\textit{Second and third authors' address}: \= \kill
\textit{First author's address}:            \> Department of Pure Mathematics \\
                                            \> University of Waterloo \\
                                            \> Waterloo, Ontario \\
                                            \> Canada N2L 3G1 \\[\medskipamount]
\textit{First author's e-mail}:             \> \texttt{beforres@uwaterloo.ca} \\[\bigskipamount]                             
\textit{Second and third authors' address}: \> Department of Mathematical and Statistical Sciences \\
                                            \> University of Alberta \\
                                            \> Edmonton, Alberta \\
                                            \> Canada T6G 2G1 \\[\medskipamount]
\textit{Second named author's e-mail}:      \> \texttt{vrunde@ualberta.ca}\\[\smallskipamount]
\textit{Third named author's e-mail}:       \> \texttt{schlitt@ualberta.ca}\\[\smallskipamount]     
\end{tabbing}
\end{document}